\documentclass{amsart}
\usepackage{amsmath,amssymb}
\usepackage[dvips]{graphics}
\usepackage[dvipdfmx]{graphicx}
\usepackage[all]{xy}

\newtheorem{Theorem}{Theorem}[section]
\newtheorem{Lemma}[Theorem]{Lemma}
\newtheorem{Proposition}[Theorem]{Proposition}

\theoremstyle{definition}

\newtheorem{Problem}[Theorem]{Problem}

\theoremstyle{remark}
\newtheorem{Remark}[Theorem]{Remark}

\def\labelenumi{\rm ({\makebox[0.8em][c]{\theenumi}})}
\def\theenumi{\arabic{enumi}}
\def\labelenumii{\rm {\makebox[0.8em][c]{(\theenumii)}}}
\def\theenumii{\roman{enumii}}

\renewcommand{\thefigure}
{\arabic{section}.\arabic{figure}}
\makeatletter
\@addtoreset{figure}{section}
\def\@thmcountersep{-}
\makeatother

\numberwithin{equation}{section}

\begin{document} 

\title[An intrinsically linked simplicial $n$-complex]{An intrinsically linked simplicial $n$-complex}

\author{Ryo Nikkuni}
\address{Department of Information and Mathematical Sciences, School of Arts and Sciences, Tokyo Woman's Christian University, 2-6-1 Zempukuji, Suginami-ku, Tokyo 167-8585, Japan}
\email{nick@lab.twcu.ac.jp}


\date{}



\begin{abstract}
For any positive integer $n$, Segal--Spie\.{z}, Lov\'{a}sz--Schrijver, Taniyama and Skopenkov provided examples of simplicial $n$-complexes that inevitably contain a nonsplittable two-component link of $n$-spheres, no matter how they are embedded into the Euclidean $(2n+1)$-space. In this paper, we introduce a new example of such a simplicial $n$-complex through a simple argument in piecewise linear topology and an application of the van Kampen--Flores theorem. Furthermore, we demonstrate the existence of additional such complexes through higher dimensional generalizations of the $\triangle Y$-exchange on graphs. 
\end{abstract}

\maketitle

\section{Introduction} 

Throughout this paper, we work in the piecewise linear category. We refer the reader to \cite{H69}, \cite{RS82} for the fundamentals of piecewise linear topology. Let $K$ be a finite simplicial $n$-complex, which we identify with its polyhedron in this context. It is well-known that every simplicial $n$-complex can be embedded in ${\mathbb R}^{2n+1}$. For an embedding $f$ of $K$ into $\mathbb{R}^{2n+1}$, we consider the image $f(K)$ up to ambient isotopy, where for two embeddings $f$ and $g$ of $K$ into $\mathbb{R}^{2n+1}$, the images $f(K)$ and $g(K)$ are said to be \textit{ambient isotopic} if there exists an orientation-preserving self-homeomorphism $h$ on $\mathbb{R}^{2n+1}$ such that $h(f(K)) = g(K)$. Let $\Lambda^{n}(K)$ be the set of all unordered pairs of mutually disjoint two subcomplexes of $K$, each of which is homeomorphic to an $n$-sphere. We identify any pair $(\gamma_{1},\gamma_{2})$ in $\Lambda^{n}(K)$ with the disjoint union $\gamma_{1}\sqcup \gamma_{2}$. Then for any pair $\lambda=\gamma_{1}\sqcup\gamma_{2}$ in $\Lambda^{n}(K)$, the image $f(\lambda)=f(\gamma_{1})\sqcup f(\gamma_{2})$ forms a two-component link of $n$-spheres in $f(K)$.

For a $2$-component link $L=K_{1}\sqcup K_{2}$ of $n$-spheres in ${\mathbb R}^{2n+1}$, the \textit{${\mathbb Z}_{2}$-linking number} ${\rm lk}_{2}(L)={\rm lk}_{2}(K_{1},K_{2})={\rm lk}_{2}(K_{2},K_{1})\in {\mathbb Z}_{2}$ is well-defined (cf. \cite[pp. 132--136]{Rolfsen76}). In particular, in the case of $n=1$, the following result is well-known as the \textit{Conway--Gordon--Sachs theorem}.

\begin{Theorem}\label{CGS} 
{\rm (Conway--Gordon \cite{CG83}, Sachs \cite{S84})} For any embedding $f$ of the complete graph on six vertices $K_{6}$ into ${\mathbb R}^{3}$, there exists a pair $\lambda$ in $\Lambda^{1}(K_{6})$ such that ${\rm lk}_{2}(f(\lambda))=1$. 
\end{Theorem}

A graph $G$ is said to be \textit{intrinsically linked} if for every embedding $f$ of $G$ into ${\mathbb R}^{3}$, the image $f(G)$ contains a nonsplittable link. Theorem \ref{CGS} implies that $K_{6}$ is intrinsically linked. In addition, in \cite{S84}, it was shown that the complete tripartite graph $K_{3,3,1}$, also denoted $P_7$, is intrinsically linked. It was also pointed out that a total of seven graphs as depicted in Fig. \ref{Petersen}, obtained from $K_{6}$ or $P_{7}$ by a finite sequence of $\triangle Y$-exchanges, are also intrinsically linked. Here, a {\it $\triangle Y$-exchange} is an operation that transforms a graph $G_{\triangle}$ into a new graph $G_{Y}$ by removing all edges of a 3-cycle $\triangle$ in $G_{\triangle}$ with edges $uv$, $vw$, and $wu$, and adding a new vertex $x$ connected to each of the vertices $u$, $v$, and $w$, as shown in Fig. \ref{Delta-Y}. This operation preserves the intrinsic linkedness for graphs. The set of these seven graphs is known as the \textit{Petersen family}, with $P_{10}$ specifically referred to as the \textit{Petersen graph}. Furthermore, Robertson--Seymour--Thomas showed in \cite{RST95} that a graph is intrinsically linked if and only if it contains a graph in the Petersen family as a \textit{minor} (see Remark \ref{minor}).

\begin{figure}[htbp]
\begin{center}
\scalebox{0.525}{\includegraphics*{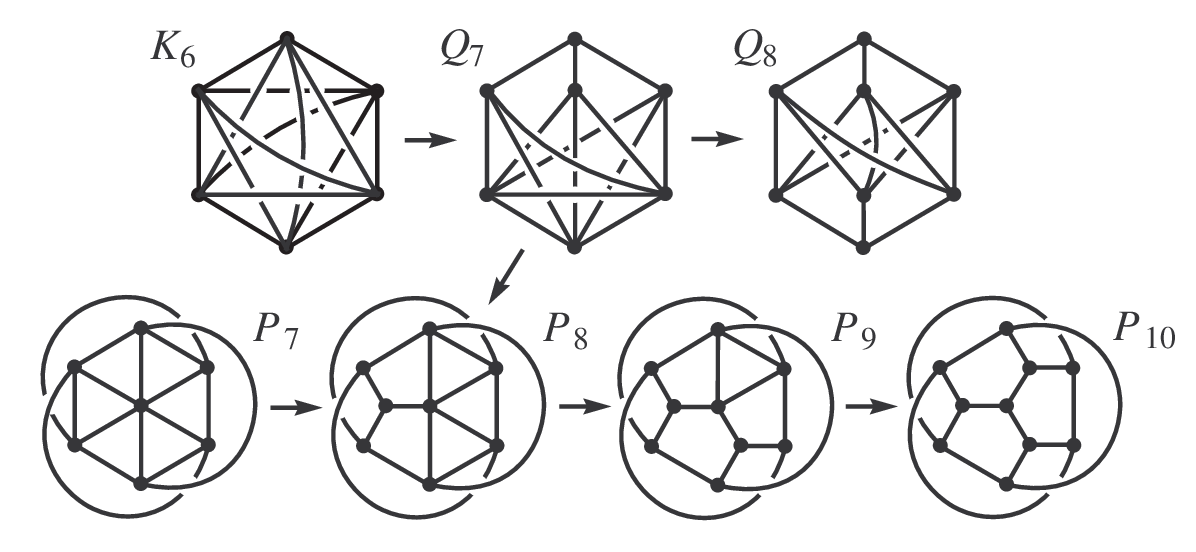}}
\caption{Petersen family (Each arrow represents a $\triangle Y$-exchange)}
\label{Petersen}
\end{center}
\end{figure}

\begin{figure}[htbp]
\begin{center}
\scalebox{0.425}{\includegraphics*{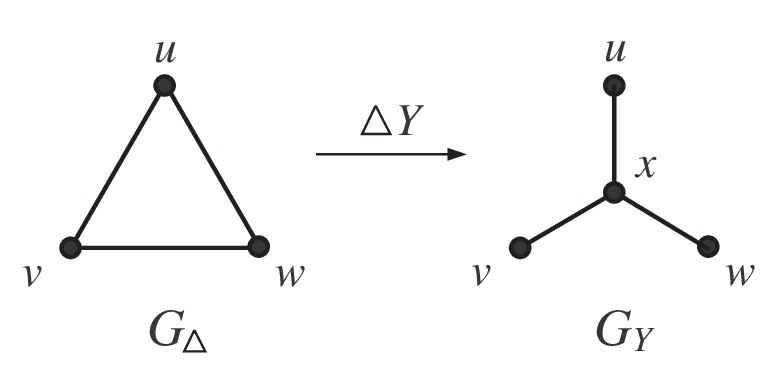}}
\caption{$\triangle Y$-exchange}
\label{Delta-Y}
\end{center}
\end{figure}

On the other hand, several results analogous to Theorem \ref{CGS} are known in higher dimensions. For any positive integer $n$, let $\sigma_{m}^{n}$ denote the $n$-skeleton of an $m$-simplex, and let $[k]^{*n+1}$ represent the $(n+1)$-fold join of $k$ points. Note that $[k]^{*n+1}$ can also be naturally regarded as a simplicial $n$-complex. Then the following results are known:

\begin{Theorem}\label{oldIL} 
\begin{enumerate}
\item {\rm (Segal--Spie\.{z} \cite{SeSp92}, Lov\'{a}sz--Schrijver \cite{LS98}, Taniyama \cite{T00})} For any embedding $f$ of $\sigma_{2n+3}^{n}$ into ${\mathbb R}^{2n+1}$, there exists a pair $\lambda$ in $\Lambda^{n}(\sigma_{2n+3}^{n})$ such that ${\rm lk}_{2}(f(\lambda))=1$. 
\item {\rm (Skopenkov \cite{Sk03})} For any embedding $f$ of $[4]^{*n+1}$ into ${\mathbb R}^{2n+1}$, there exists a pair $\lambda$ in $\Lambda^{n}([4]^{*n+1})$ such that ${\rm lk}_{2}(f(\lambda))= 1$. 
\end{enumerate}
\end{Theorem}

We also refer the reader to \cite{KS}, \cite{Tu13} for related works. Theorem \ref{oldIL} implies that for any positive integer $n$, there exist simplicial $n$-complexes that inevitably contain a nonsplittable two-component link of $n$-spheres, no matter how they are embedded into $\mathbb{R}^{2n+1}$. Notably, $\sigma_{5}^{1}$ corresponds to the complete graph on $6$ vertices $K_{6}$, and $[4]^{*2}$ corresponds to the complete bipartite graph on $4+4$ vertices $K_{4,4}$, which contains $K_{3,3,1}=P_{7}$ as a proper minor. Our first purpose in this paper is to present another simplicial $n$-complex whose embeddings into $\mathbb{R}^{2n+1}$ always contain a nonsplittable link of $n$-spheres. For a positive integer $n$, let us define the specific simplicial $n$-complex $K^{(n)}$ abstractly as follows. For an $m$-simplex $\sigma_{m}=|a_{0}a_{1}\cdots a_{m}|$ where $a_{0},a_{1},\ldots,a_{m}$ are the $0$-simplices of $\sigma_{m}$, we denote the simplicial $m$-complex derived from $\sigma_{m}$ by $K(\sigma_{m})=K(a_{0}a_{1}\cdots a_{m})$. Consider $n+1$ mutually disjoint sets $V^{i}$, each consisting of three $0$-simplices $a_{0}^{i},a_{1}^{i}$ and $a_{2}^{i}$ ($i=0,1,\ldots,n$). Let $b$ be a $0$-simplex that is not an element in any $V^i$. Then we consider the following two types of $n$-simplices: 

\vspace{0.2cm}
\noindent
{\bf Type I.} 
The $n$-simplex $|a_{j_{0}}^{0}a_{j_{1}}^{1}\cdots a_{j_{n}}^{n}|$ spanned by the $0$-simplices $a_{j_{i}}^{i}$, one chosen from each $V^{i}$ for $i=0,1,\ldots,n$ and $j_{i}\in \{0,1,2\}$.

\vspace{0.2cm}
\noindent
{\bf Type II.} 
The $n$-simplex $|ba_{j_{0}}^{0}\cdots \hat{a}_{j_{q}}^{q}\cdots a_{j_{n}}^{n}|$ spanned by the $0$-simplices $a_{j_{i}}^{i}$, one chosen from each of $n$ of the $n+1$ sets $V^i$ for $i=0,1,\ldots,n$ and $j_{i}\in \{0,1,2\}$, and $b$. Here, $\hat{a}_{j_{q}}^{q}$ denotes the omission of $a_{j_{q}}^{q}$ for $q\in \{0,1,\ldots,n\}$.

\vspace{0.2cm}
\noindent
See Fig. \ref{K331complex} for each type of simplex if $n=2$. We now consider all $n$-simplices of Types I and II, and define $K^{(n)}$ as the union of all simplicial complexes derived from them. Namely we define 
\begin{eqnarray*}
K^{(n)} = \bigcup_{\substack{q\in\{0,1,2,\ldots,n\} \\ j_{i}\in\{0,1,2\}}}K\big(ba_{j_{0}}^{0}\cdots \hat{a}_{j_{q}}^{q}\cdots a_{j_{n}}^{n}\big)
\cup \bigcup_{j_{i}\in\{0,1,2\}} K\big(a_{j_{0}}^{0}a_{j_{1}}^{1}\cdots a_{j_{n}}^{n}\big). 
\end{eqnarray*}
In particular, $K^{(1)}$ corresponds to $K_{3,3,1}=P_{7}$ in the Petersen family, see Fig. \ref{K331P7}. Namely, while $\sigma_{2n+3}^{n}$ generalizes $K_{6}$ to higher dimensions, $K^{(n)}$ serves as a higher dimensional analogue of $K_{3,3,1}$. We also remark here that the union of all of the simplicial complexes obtained from $n$-simplices of Type II, $\bigcup_{j_{i}\in\{0,1,2\}} K\big(a_{j_{0}}^{0}a_{j_{1}}^{1}\cdots a_{j_{n}}^{n}\big)$, is isomorphic to $[3]^{*n+1}$. We denote this subcomplex of $K^{(n)}$ by $H^{(n)}$. For example, $H^{(1)}$ is none other than $K_{3,3}$. Then we have the following.

\begin{figure}[htbp]
\begin{center}
\scalebox{0.45}{\includegraphics*{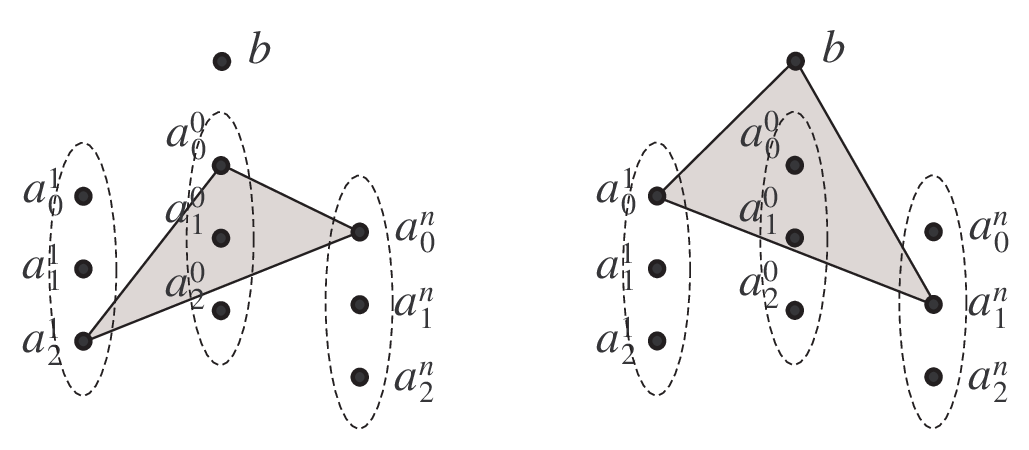}}
\caption{$|a_{0}^{0}a_{2}^{1}a_{0}^{n}|$ of Type I and $|ba_{0}^{1}a_{1}^{n}|$ of Type II ($n=2$)}
\label{K331complex}
\end{center}
\end{figure}

\begin{figure}[htbp]
\begin{center}
\scalebox{0.45}{\includegraphics*{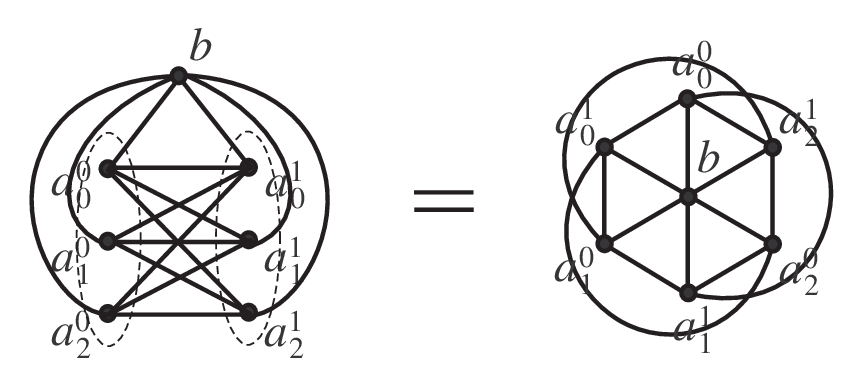}}
\caption{$K^{(1)}=K_{3,3,1}=P_{7}$}
\label{K331P7}
\end{center}
\end{figure}

\begin{Theorem}\label{newIL} 
Let $n$ be a positive integer. For every embedding $f$ of $K^{(n)}$ into ${\mathbb R}^{2n+1}$, the following holds: 
\begin{eqnarray*}
\sum_{\lambda\in \Lambda^{n}(K^{(n)})}{\rm lk}_{2}(f(\lambda))\equiv 1\pmod{2}. 
\end{eqnarray*}
\end{Theorem}

Theorem \ref{newIL} implies that there exists a pair $\lambda$ in $\Lambda^{n}(K^{(n)})$ such that $f(\lambda)$ is a nonsplittable two-component link of $n$-spheres. Here, each pair in $\Lambda^{n}(K^{(n)})$ consists of a subcomplex isomorphic to the boundary of an $(n+1)$-simplex and a subcomplex isomorphic to $[2]^{*n+1}$, see Fig. \ref{tetraocta} for the case $n=2$. In this paper, we refer to the former as an \textit{$n$-tetrahedron} and the latter as an \textit{$n$-octahedron}. Theorem \ref{newIL} states that in any embedding of $K^{(n)}$ into $\mathbb{R}^{2n+1}$, there exists a pair consisting of an $n$-tetrahedron and an $n$-octahedron that is linked. In Section 2, we prove Theorem \ref{newIL} using a novel approach that combines a simple  argument in piecewise linear topology with an application of the van Kampen--Flores theorem, which ensures the non-embeddability of certain simplicial $n$-complexes into $\mathbb{R}^{2n}$.

\begin{figure}[htbp]
\begin{center}
\scalebox{0.5}{\includegraphics*{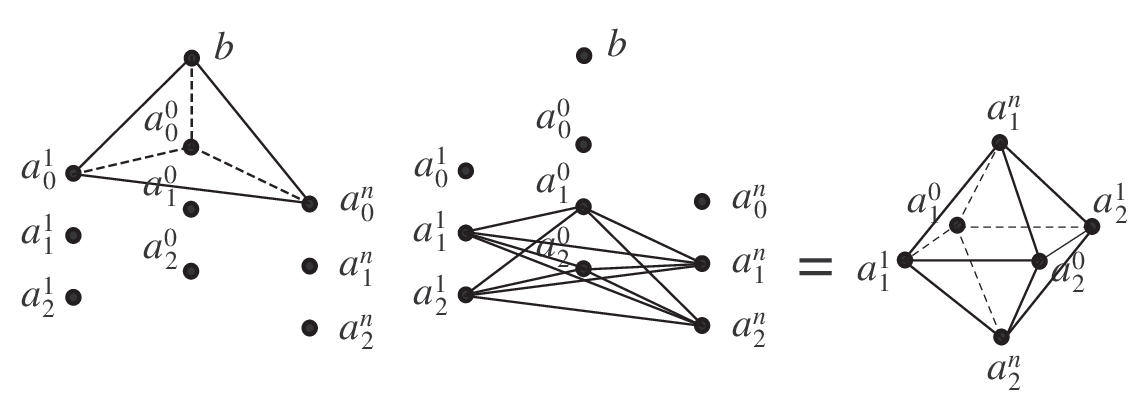}}
\caption{$n$-tetrahedron, $n$-octahedron ($n=2$)}
\label{tetraocta}
\end{center}
\end{figure}

\begin{Remark}\label{minor}
A graph $H$ is called a \textit{minor} of a graph $G$ if there exists a subgraph $G'$ of $G$ such that $H$ is obtained from $G'$ by a finite sequence of edge contractions. In particular, a minor $H$ of $G$ is called a \textit{proper minor} of $G$ if $H\neq G$. In our high dimensional case, we can see that there exists an $n$-subcomplex $L$ of $[4]^{*n+1}$ such that $K^{(n)}$ is obtained from $L$ by contracting exactly one $n$-simplex (imagine that in the case of  $n=1$, $K_{3,3,1}$ is obtained as a proper minor of $K_{4,4}$). Therefore, $K^{(n)}$ can be regarded as a higher dimensional proper minor of $[4]^{*n+1}$. 
\end{Remark}

Our second purpose in this paper is to generalize the $\triangle Y$-exchange to higher dimensions and obtain numerous simplicial $n$-complexes that are `intrinsically linked'. Let $K_{{\triangle}^{n}}$ be a simplicial $n$-complex containing an $n$-subcomplex $\triangle^{n}$ that is isomorphic to the boundary of an $(n+1)$-simplex $\sigma_{n+1}=|a_{0}a_{1}\cdots a_{n+1}|$. We identify $\triangle^{n}$ with $\partial \sigma_{n+1}$. Then, consider the join of the $(n-1)$-skeleton $\sigma_{n+1}^{n-1}$ of $\sigma_{n+1}$ with another $0$-simplex $x$ disjoint from $\sigma_{n+1}$, and let $K_{Y^{n}}$ denote the simplicial $n$-complex obtained by replacing $\triangle^{n}$ with $Y^{n}=\sigma_{n+1}^{n-1}*x$. We call this operation, which transforms $K_{\triangle^{n}}$ into $K_{Y^{n}}$, the \textit{$\triangle Y(n)$-exchange}. See Fig. \ref{Delta-Y_n} for $n=2$. Note that the $\triangle Y(1)$-exchange corresponds to the $\triangle Y$-exchange on graphs. Then we have the following.

\begin{figure}[htbp]
\begin{center}
\scalebox{0.48}{\includegraphics*{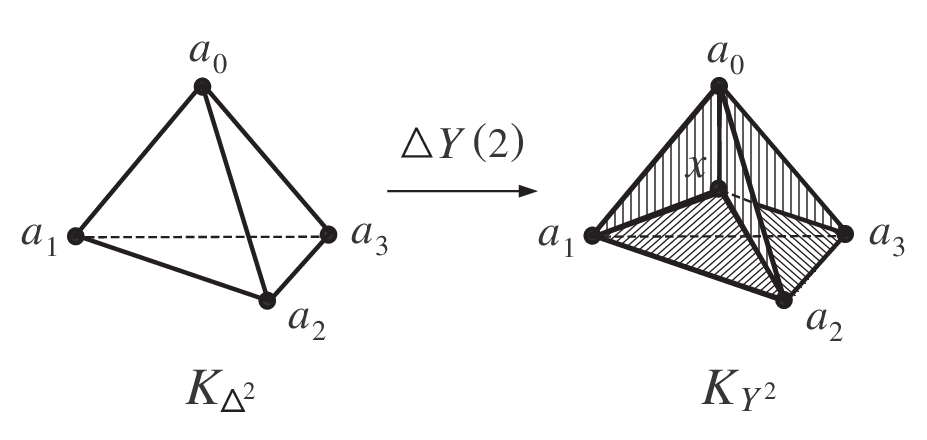}}
\caption{$\triangle Y(n)$-exchange ($n=2$)}
\label{Delta-Y_n}
\end{center}
\end{figure}

\begin{Theorem}\label{deltayil}
If for any embedding $f'$ of $K_{\triangle^{n}}$ into $\mathbb{R}^{2n+1}$ there exists a pair $\lambda'$ in $\Lambda^{n}(K_{\triangle^{n}})$ such that ${\rm lk}_{2}(f'(\lambda'))=1$, then for any embedding $f$ of $K_{Y^{n}}$ into $\mathbb{R}^{2n+1}$ there exists a pair $\lambda$ in $\Lambda^{n}(K_{Y^{n}})$ such that ${\rm lk}_{2}(f(\lambda))=1$. 
\end{Theorem}

By applying Theorem \ref{deltayil} to Theorem \ref{oldIL} (1) and Theorem \ref{newIL}, we can construct numerous simplicial $n$-complexes that inevitably contain a nonsplittable two-component link of $n$-spheres, no matter how they are embedded into $\mathbb{R}^{2n+1}$. We prove Theorem \ref{deltayil} in Section 3 and discuss an additional noteworthy `intrinsically linked' simplicial $n$-complex.

\section{van Kampen--Flores theorem and a proof of Theorem \ref{newIL}} 

In proving Theorem \ref{newIL}, let us recall the so-called \textit{van Kampen--Flores theorem}. Let $K$ be a simplicial $n$-complex. Then it is also well-known that $K$ can be generically immersed into $\mathbb{R}^{2n}$, where an immersion $\varphi$ of $K$ into $\mathbb{R}^{2n}$ is said to be \textit{generic} if all singularities of $\varphi(K)$ are transversal double points occurring between the interiors of pairs of $n$-simplices. For an integer $k\le n$, let ${\varDelta}^{k}(K)$ denote the set of all $k$-simplices in $K$. For a generic immersion $\varphi$ of $K$ into $\mathbb{R}^{2n}$ and a pair of mutually disjoint $n$-simplices $\sigma$ and $\tau$ in $\varDelta^{n}(K)$, we denote the number of all double points occurring between $\varphi(\sigma)$ and $\varphi(\tau)$ by $l(\varphi(\sigma),\varphi(\tau))$. Then the following result is known:

\begin{Theorem}{\rm (van Kampen \cite{VK33}, Flores \cite{Flores32})}\label{VKF} 
Let $n$ be a positive integer. Let $K$ be the $n$-skeleton of a $(2n+2)$-simplex $\sigma_{2n+2}^{n}$ or the $(n+1)$-fold join of $3$ points $[3]^{*n+1}$. Then for every generic immersion $\varphi$ of $K$ into ${\mathbb R}^{2n}$, the following holds: 
\begin{eqnarray*}
\sum_{\substack{\sigma,\tau\in\varDelta^{n}(K)\\ \sigma\cap\tau=\emptyset}}l(\varphi(\sigma),\varphi(\tau))\equiv 1\pmod{2}. 
\end{eqnarray*}
\end{Theorem}

\begin{Remark}\label{K5K33}
Theorem \ref{VKF} implies that both $\sigma_{2n+2}^n$ and $[3]^{*n+1}$ cannot be embedded in $\mathbb{R}^{2n}$. In particular, for $n=1$, this corresponds to the classical fact that both $K_{5}$ and $K_{3,3}$ cannot be embedded in $\mathbb{R}^{2}$. 
\end{Remark}

Let $f$ be an embedding of $K^{(n)}$ into $\mathbb{R}^{2n+1}$. Let $\pi$ be a natural projection from $\mathbb{R}^{2n+1}$ to $\mathbb{R}^{2n}$ defined by $\pi(x_{1},x_{2},\ldots,x_{2n},x_{2n+1}) = (x_{1},x_{2},\ldots,x_{2n})$. We denote the composition map $\pi\circ f$ from $K^{(n)}$ to $\mathbb{R}^{2n}$ by $\hat{f}$. Then, by perturbing $f(K^{(n)})$ up to ambient isotopy if necessary, we may assume that $\hat{f}$ is a generic immersion of $K^{(n)}$ into $\mathbb{R}^{2n}$. Then we have the following. 

\begin{Lemma}\label{mainlem}
Let $n$ be a positive integer. For every embedding $f$ of $K^{(n)}$ into ${\mathbb R}^{2n+1}$, the following holds: 
\begin{eqnarray*}
\sum_{\lambda\in \Lambda^{n}(K^{(n)})}{\rm lk}_{2}(f(\lambda))
\equiv 
\sum_{\substack{\sigma,\tau\in\varDelta^{n}(H^{(n)})\\ \sigma\cap\tau=\emptyset}}l(\hat{f}(\sigma),\hat{f}(\tau))\pmod{2}. 
\end{eqnarray*}
\end{Lemma}

\begin{proof}
For a pair of mutually disjoint $n$-simplices $\sigma$ and $\tau$ in $K^{(n)}$, let $\omega(\hat{f}(\sigma), \hat{f}(\tau))$ denote the number of all double points where $\hat{f}(\tau)$ crosses over $\hat{f}(\sigma)$ with respect to the projection $\pi$. Let $\Gamma^{1}$ denote the set of all $n$-tetrahedra in $K^{(n)}$, and let $\Gamma^{2}$ denote the set of all $n$-octahedra in $K^{(n)}$. Then any pair $\lambda$ in $\Lambda^{n}(K^{(n)})$ consists of a pair of an $n$-tetrahedron $\gamma_{1}$ in $\Gamma^{1}$ and an $n$-octahedron $\gamma_{2}$ in $\Gamma^{2}$ that are mutually disjoint. Then, the $\mathbb{Z}_{2}$-linking number of the two-component link $f(\lambda)$ is calculated as follows: 
\begin{eqnarray}\label{lkomega}
{\rm lk}_{2}(f(\lambda)) = {\rm lk}_{2}(f(\gamma_{1}),f(\gamma_{2})) 
\equiv\sum_{\substack{\sigma,\tau\in \varDelta^{n}(\lambda)\\ \sigma\subset\gamma_{1},\,\tau\subset\gamma_{2}}}\omega(\hat{f}(\sigma),\hat{f}(\tau))\pmod{2}. 
\end{eqnarray}
Thus by (\ref{lkomega}), we have 
\begin{eqnarray}\label{lksum}
\sum_{\lambda\in \Lambda^{n}(K^{(n)})}{\rm lk}_{2}(f(\lambda))
&=& \sum_{\substack{\lambda=\gamma_{1}\sqcup\gamma_{2} \\ \gamma_{1}\in \Gamma^{1},\,\gamma_{2}\in\Gamma^{2}}}{\rm lk}_{2}(f(\gamma_{1}),f(\gamma_{2}))\\
&=& \sum_{\substack{\lambda=\gamma_{1}\sqcup\gamma_{2} \\ \gamma_{1}\in \Gamma^{1},\,\gamma_{2}\in\Gamma^{2}}}
\bigg(\sum_{\substack{\sigma,\tau\in \varDelta^{n}(\lambda)\\ \sigma\subset\gamma_{1},\,\tau\subset\gamma_{2}}}\omega(\hat{f}(\sigma),\hat{f}(\tau))\bigg).\nonumber
\end{eqnarray}

Here, in $\omega(\hat{f}(\sigma),\hat{f}(\tau))$ appearing on the right side of (\ref{lksum}), the pair of mutually disjoint $n$-simplies $\sigma$ and $\tau$ can be one of the following two cases: 

\vspace{0.2cm}
\noindent
\textbf{Case 1.} $\sigma$ is an $n$-simplex containing $b$, and $\tau$ is an $n$-simplex in $\varDelta^{n}(H^{(n)})$. 

\vspace{0.2cm}
\noindent
\textbf{Case 2.}  Both $\sigma$ and $\tau$ are $n$-simplices in $\varDelta^{n}(H^{(n)})$. 

\vspace{0.2cm}
\noindent 
First, in Case 1, it suffices to consider that $\sigma = |ba_{0}^{0}\cdots a_{0}^{n-1}|$ and $\tau = |a_{1}^{0}a_{1}^{1}\cdots a_{1}^{n}|$. Then there exist exactly two pairs $\lambda=\gamma_{1}\sqcup \gamma_{2}$ and $\lambda'=\gamma'_{1}\sqcup \gamma'_{2}$ in $\Lambda^{n}(K^{(n)})$ such that $\sigma$ and $\tau$ belong to separate components, where 
\begin{eqnarray*}
&&\gamma_{1}=\partial |ba_{0}^{0}\cdots a_{0}^{n-1}a_{0}^{n}|,\ 
\gamma_{2}=\{a_{1}^{0},a_{2}^{0}\}*\{a_{1}^{1},a_{2}^{1}\}*\cdots * \{a_{1}^{n-1},a_{2}^{n-1}\}*\{a_{1}^{n},a_{2}^{n}\},\\
&&\gamma'_{1}=\partial |ba_{0}^{0}\cdots a_{0}^{n-1}a_{2}^{n}|,\ 
\gamma'_{2}=\{a_{1}^{0},a_{2}^{0}\}*\{a_{1}^{1},a_{2}^{1}\}*\cdots * \{a_{1}^{n-1},a_{2}^{n-1}\}*\{a_{0}^{n},a_{1}^{n}\}. 
\end{eqnarray*}
See Fig. \ref{n_tetraocta1} (1) for $n=2$. Since $\sigma$ is contained in both $\gamma_{1}$ and $\gamma'_{1}$, and $\tau$ is contained in both $\gamma_{2}$ and $\gamma'_{2}$, the term $2\omega(\hat{f}(\sigma),\hat{f}(\tau))$ appears on the right side of (\ref{lksum}) and vanishes modulo $2$. 
Next, in Case 2, it suffices to consider that $\sigma = |a_{0}^{0}a_{0}^{1}\cdots a_{0}^{n}|$ and $\tau = |a_{1}^{0}a_{1}^{1}\cdots a_{1}^{n}|$. Then there exist exactly two pairs $\lambda=\gamma_{1}\sqcup \gamma_{2}$ and $\lambda'=\gamma'_{1}\sqcup \gamma'_{2}$ in $\Lambda^{n}(K^{(n)})$ such that $\sigma$ and $\tau$ belong to separate components, where 
\begin{eqnarray*}
&&\gamma_{1}=\partial |ba_{0}^{0}a_{0}^{1}\cdots a_{0}^{n}|,\ 
\gamma_{2}=\{a_{1}^{0},a_{2}^{0}\}*\{a_{1}^{1},a_{2}^{1}\}*\cdots * \{a_{1}^{n-1},a_{2}^{n-1}\}*\{a_{1}^{n},a_{2}^{n}\},\\
&&\gamma'_{1}=\partial |ba_{1}^{0}a_{1}^{1}\cdots a_{1}^{n}|,\ 
\gamma'_{2}=\{a_{0}^{0},a_{2}^{0}\}*\{a_{0}^{1},a_{2}^{1}\}*\cdots * \{a_{0}^{n-1},a_{2}^{n-1}\}*\{a_{0}^{n},a_{2}^{n}\}. 
\end{eqnarray*}
See Fig. \ref{n_tetraocta1} (2) for $n=2$. Since $\sigma$ is contained in both $\gamma_{1}$ and $\gamma'_{2}$, and $\tau$ is contained in both $\gamma'_{1}$ and $\gamma_{2}$, the term $\omega(\hat{f}(\sigma),\hat{f}(\tau)) + \omega(\hat{f}(\tau),\hat{f}(\sigma))$ appears on the right side of (\ref{lksum}) and is equal to $l(\hat{f}(\tau),\hat{f}(\sigma))$. Hence we have 
\begin{eqnarray}\label{lksum2}
\sum_{\substack{\lambda=\gamma_{1}\sqcup\gamma_{2} \\ \gamma_{1}\in \Gamma^{1},\,\gamma_{2}\in\Gamma^{2}}}
\bigg(\sum_{\substack{\sigma,\tau\in \varDelta^{n}(\lambda)\\ \sigma\subset\gamma_{1},\,\tau\subset\gamma_{2}}}\omega(\hat{f}(\sigma),\hat{f}(\tau))\bigg)
&\equiv& \sum_{\substack{\sigma,\tau\in\varDelta^{n}(H^{(n)})\\ \sigma\cap\tau=\emptyset}}l(\hat{f}(\sigma),\hat{f}(\tau))\pmod{2}. 
\end{eqnarray}
By (\ref{lksum}) and (\ref{lksum2}), we have the result. 
\end{proof}

\begin{figure}[htbp]
\begin{center}
\scalebox{0.525}{\includegraphics*{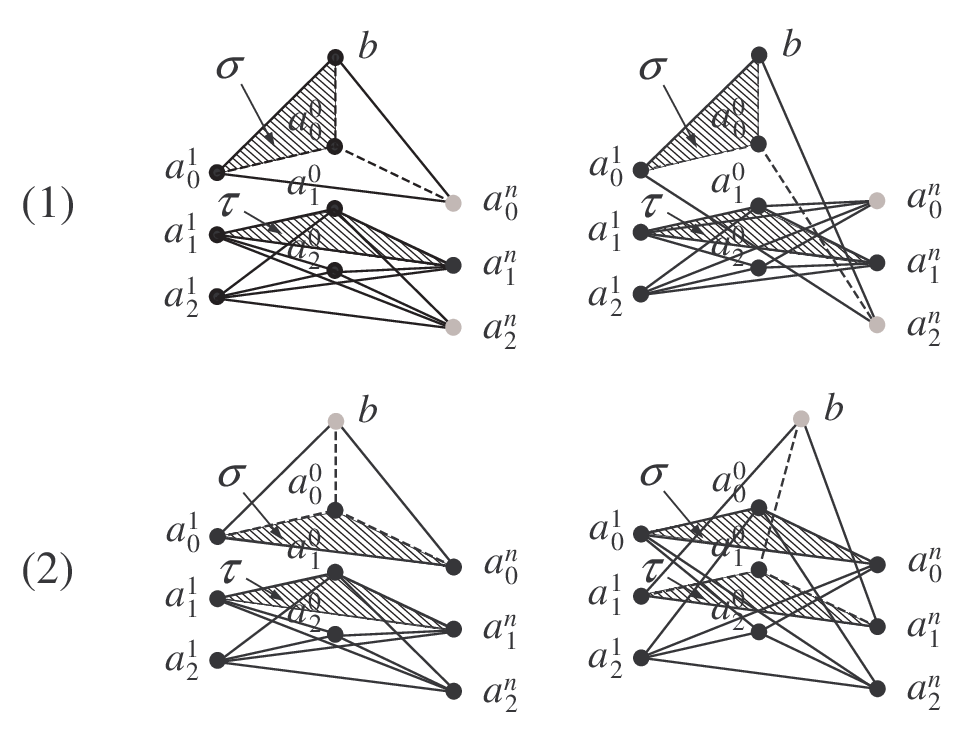}}
\caption{Exactly two pairs in $\Lambda^{n}(K^{(n)})$ containing $\sigma$ and $\tau$ in separate components ($n=2$)}
\label{n_tetraocta1}
\end{center}
\end{figure}

\begin{proof}[Proof of Theorem \ref{newIL}]
Recall that $H^{(n)}$ is isomorphic to $[3]^{*n+1}$. Therefore, Theorem \ref{newIL} follows directly from Lemma \ref{mainlem} and Theorem \ref{VKF}. 
\end{proof}

\begin{Remark}
Let $\sigma_{2n+3}$ be the $(2n+3)$-simplex $|a_{0}a_{1}\cdots a_{2n+3}|$ and $\sigma_{2n+3}^{n}$ its $n$-skeleton. Note that $\sigma_{2n+3}^{n}$ contains the $n$-skeleton of the $(2n+2)$-face $|a_{1}a_{2}\cdots a_{2n+3}|$ of $\sigma_{2n+3}$ as an $n$-subcomplex. Let $\Gamma^{1}$ denote the set of all $n$-tetrahedra in $\sigma_{2n+3}^{n}$ that contain $a_{0}$, and let $\Gamma^{2}$ denote the set of all $n$-tetrahedra in $\sigma_{2n+3}^{n}$ that do not contain $a_{0}$. Then any pair $\lambda$ in $\Lambda^{n}(\sigma_{2n+3}^{n})$ consists of a pair of an $n$-tetrahedron $\gamma_{1}$ in $\Gamma^{1}$ and an $n$-tetrahedron $\gamma_{2}$ in $\Gamma^{2}$ that are mutually disjoint.  Then, for any embedding $f$ of $\sigma_{2n+3}^n$ into $\mathbb{R}^{2n+1}$, in a similar way as the proof of Theorem \ref{newIL}, we can obtain the following: 
\begin{eqnarray*}
\sum_{\lambda\in \Lambda^{n}(\sigma_{2n+3}^n)}{\rm lk}_{2}(f(\lambda))
\equiv 1\pmod{2}.  
\end{eqnarray*}
This provides an alternative proof of Theorem \ref{oldIL} (1). 
\end{Remark}

\section{Higher dimensional $\triangle Y$-exchange} 

Let $K_{\triangle^{n}}$ and $K_{Y^{n}}$ be two simplicial $n$-complex such that $K_{Y^{n}}$ is obtained from $K_{\triangle^{n}}$ by a single $\triangle Y(n)$-exchange, as described in Section 1. Let $\Lambda_{\triangle^{n}}^{n}(K_{\triangle^{n}})$ denote the set of all pairs in $\Lambda^{n}(K_{\triangle^{n}})$ containing $\triangle^{n}$ as a component. Let $\lambda'$ be a pair in $\Lambda^{n}(K_{\triangle^{n}})$ that does not contain $\triangle^{n}$. Then we can see that there exists a pair $\Phi(\lambda')$ in $\Lambda^{n}(K_{Y^{n}})$ such that $\lambda'\setminus \triangle^{n} = \Phi(\lambda')\setminus Y^{n}$, and the correspondence from $\lambda'$ to $\Phi(\lambda')$ defines a surjective map $\Phi:\Lambda^{n}(K_{\triangle^{n}})\setminus \Lambda_{\triangle^{n}}^{n}(K_{\triangle^{n}})\to \Lambda^{n}(K_{Y^{n}})$.

Let $f$ be an embedding of $K_{Y^{n}}$ into $\mathbb{R}^{2n+1}$. Let $\widetilde{K}_{Y^{n}}$ be an simplicial $(n+1)$-complex defined by 
\begin{eqnarray*}
\widetilde{K}_{Y^{n}} 
= K_{Y^{n}}\cup \bigcup_{q=0}^{n+1}K(a_{0}\cdots \hat{a}_{q}\cdots a_{n+1}x). 
\end{eqnarray*}
Then by using a standard general position argument in piecewise linear topology, the embedding $f$ of $K_{Y^{n}}$ extends to an embedding $F$ of $\widetilde{K}_{Y^{n}}$ into $\mathbb{R}^{2n+1}$. Note that $\widetilde{K}_{Y^{n}}$ contains $K_{\triangle^{n}}$ as an $n$-subcomplex, and the restriction of $F$ to $K_{\triangle^{n}}$, denoted by $f'$, is an embedding of $K_{\triangle^{n}}$ into $\mathbb{R}^{2n+1}$. Then we have the following, where the case $n=1$ is shown in \cite[Proposition 2.1]{NT12}.

\begin{Lemma}\label{DeltaYiso}
Let $\lambda'$ be a pair in $\Lambda^{n}(K_{\triangle^{n}})\setminus \Lambda_{\triangle^{n}}^{n}(K_{\triangle^{n}})$. Then the link $f'(\lambda')$ is ambient isotopic to the link $f(\Phi(\lambda'))$. 
\end{Lemma}

\begin{proof}
Two links $f'(\lambda')$ and $f(\Phi(\lambda'))$ are transformed into each other by so-called simplex moves. Thus we have the result. 
\end{proof}

\begin{proof}[Proof of Theorem \ref{deltayil}]
Let $f$ be an embedding of $K_{Y^{n}}$ into $\mathbb{R}^{2n+1}$. Then for the embedding $f'$ of $K_{\triangle^{n}}$ into $\mathbb{R}^{2n+1}$, there exists, by assumption, a pair $\lambda'$ in $\Lambda^{n}(K_{\triangle^{n}})$ such that ${\rm lk}_{2}(f'(\lambda'))=1$. Since $f'(\triangle^{n})$ bounds an $(n+1)$-ball $B$ in $\mathbb{R}^{2n+1}$ with $f'(K_{\triangle^{n}})\cap B = f'(K_{\triangle^{n}})\cap \partial B = f'(\triangle^{n})$, it follows that $\lambda'$ does not contain $\triangle^{n}$ as a component. Hence, $\lambda'$ is a pair in $\Lambda^{n}(K_{\triangle^{n}})\setminus \Lambda_{\triangle^{n}}^{n}(K_{\triangle^{n}})$. By Lemma \ref{DeltaYiso}, $f'(\lambda')$ is ambient isotopic to $f(\Phi(\lambda'))$, and thus ${\rm lk}_{2}(f(\Phi(\lambda'))=1$.
\end{proof}

Let $K$ be a simplicial $n$-complex. Let $\mathcal{F}_{\triangle^{n}}(K)$ denote the set of all simplicial $n$-complexes obtained from $K$ by a finite sequence of $\triangle Y(n)$-exchanges. For example, in the case of $n=1$, $\mathcal{F}_{\triangle^{1}}(K_{6})$ and $\mathcal{F}_{\triangle^{1}}(K_{3,3,1})$ share exactly three graphs: $P_{8},\,P_{9}$ and $P_{10}$, and their union forms the Petersen family. However, for $n\ge 2$, we have the following. 

\begin{Proposition}\label{hdPet}
For $n\ge 2$, $\mathcal{F}_{\triangle^{n}}(\sigma_{2n+3}^{n})$ and $\mathcal{F}_{\triangle^{n}}(K^{(n)})$ are disjoint. 
\end{Proposition}

Before showing Proposition \ref{hdPet}, we define the degree of a simplex in a simplicial $n$-complex. Let $K$ be a simplicial $n$-complex and $\sigma$ a $k$-simplex in $\varDelta^{k}(K)$. Let ${\rm deg}_{K}^{n}\sigma$ denote the number of $n$-simplices in $\varDelta^{n}(K)$ that contain $\sigma$ as a $k$-face. In particular, if $n=1$ and $\sigma$ is a $0$-simplex, then this corresponds to the degree of a vertex in a simple graph.

\begin{proof}[Proof of Proposition \ref{hdPet}] 
Assume that $n\ge 2$. In the $\triangle Y(n)$-exchange from $K_{\triangle^{n}}$ to $K_{Y^{n}}$, the $(n-1)$-skeleton $K_{\triangle^{n}}^{n-1}$ of $K_{\triangle^{n}}$ is contained in $K_{Y^{n}}$ as a subcomplex. For an $(n-2)$-simplex $\sigma$ in $\varDelta^{n-2}(Y^{n})$, we can see that the degree satisfies $\deg_{Y^{n}}^{n}(\sigma) = \deg_{\triangle^{n}}^{n}(\sigma)=3$ if $\sigma$ does not contain $x$, and $\deg_{Y^{n}}^{n}(\sigma) =6$ if $\sigma$ contains $x$. Thus, for an $(n-2)$-simplex $\sigma$ in $\varDelta^{n-2}(K_{Y^{n}})$, we have: 
\begin{eqnarray}\label{deg}
\deg_{K_{Y^{n}}}^{n}(\sigma) =
\begin{cases}
\deg_{K_{\triangle^{n}}}^{n}(\sigma) & \text{if}\ \sigma\in\varDelta^{n-2}(K_{\triangle^{n}}^{n-1}), \\
6 & \text{otherwise}.
\end{cases}
\end{eqnarray}
Note that for every $(n-2)$-simplex $\tau$ in $\varDelta^{n-2}(\sigma_{2n+3}^{n})$, we have ${\rm deg}_{\sigma_{2n+3}^{n}}^{n}\tau = \tbinom{n+5}{2}$. On the other hand, for every $(n-2)$-simplex $\tau'$ in $\varDelta^{n-2}(K^{(n)})$, we have ${\rm deg}_{K^{(n)}}^{n}\tau' = 27$ if $\tau'$ contains $b$, and ${\rm deg}_{K^{(n)}}^{n}\tau' = 15$ otherwise. Since $\tbinom{n+5}{2}\neq 27,15$ for $n\ge 2$, it follows from (\ref{deg}) that $\mathcal{F}_{\triangle^{n}}(\sigma_{2n+3}^{n})$ and $\mathcal{F}_{\triangle^{n}}(K^{(n)})$ are disjoint. 
\end{proof} 

\begin{Problem}\label{family}
For a positive integer $n\ge 2$, list all elements of the sets $\mathcal{F}_{\triangle^{n}}(\sigma_{2n+3}^{n})$ and $\mathcal{F}_{\triangle^{n}}(K^{(n)})$. 
\end{Problem}

The author considers Problem \ref{family} to be generally challenging and thus will not explore it further here. Instead, we introduce a distinguished simplicial $n$-complex in $\mathcal{F}_{\triangle^{n}}(K^{(n)})$ that is terminal with respect to the sequence of $\triangle Y(n)$-exchanges. For the subcomplex $H^{(n)}$ of $K^{(n)}$, let $\Xi(H^{(n)})$ denote the subset of $\varDelta^{n}(H^{(n)})$ defined by 
\begin{eqnarray*}
\Xi(H^{(n)}) = \Big\{|a_{j_{0}}^{0}a_{j_{1}}^{1}\cdots a_{j_{n}}^{n}|\ \big|\ \sum_{i=0}^{n}j_{i}\equiv 0\pmod{3}\Big\}. 
\end{eqnarray*}
This set consists of $3^{n}$ $n$-simplices. For any two distinct $n$-simplices $|a_{j_{0}}^{0}a_{j_{1}}^{1}\cdots a_{j_{n}}^{n}|$ and $|a_{k_{0}}^{0}a_{k_{1}}^{1}\cdots a_{k_{n}}^{n}|$ in $\varDelta^{n}(H^{(n)})$, the $n$-simplex $|a_{j_0+k_{0}}^{0}a_{j_{1}+k_{1}}^{1}\cdots a_{j_{n}+k_{n}}^{n}|$ also belongs to $\Xi(H^{(n)})$, where each $j_{q}+k_{q}$ is taken modulo $3$. Furthermore, if $|a_{j_{0}}^{0}a_{j_{1}}^{1}\cdots a_{j_{n}}^{n}|$ and $|a_{k_{0}}^{0}a_{k_{1}}^{1}\cdots a_{k_{n}}^{n}|$ share an $(n-1)$-face, then there exists a unique $q$ such that $j_{q}+k_{q}\neq 0$, and $j_{i}+k_{i}=0$ for all $i\neq q$. This implies that $|a_{j_0+k_{0}}^{0}a_{j_{1}+k_{1}}^{1}\cdots a_{j_{n}+k_{n}}^{n}|$ does not belong to $\Xi(H^{(n)})$, which leads to a contradiction. Therefore, $\Xi(H^{(n)})$ is a set of $n$-simplices in $H^{(n)}$ such that no two of them share any $(n-1)$-faces. We then define the subset $\Gamma_{\Xi}^{1}$ of $\Gamma^{1}$ as  
\begin{eqnarray*}
\Gamma_{\Xi}^{1} = \{\partial|ba_{j_{0}}^{0}a_{j_{1}}^{1}\cdots a_{j_{n}}^{n}|\ |\ |a_{j_{0}}^{0}a_{j_{1}}^{1}\cdots a_{j_{n}}^{n}|\in \Xi(H^{(n)})\}. 
\end{eqnarray*}
Since no two $n$-tetrahedra in $\Gamma_{\Xi}^{1}$ share an $n$-simplex, we can apply $3^{n}$ $\triangle Y(n)$-exchanges to $K^{(n)}$ sequentially. In particular, let $P^{(n)}$ denote the simplicial $n$-complex obtained from $K^{(n)}$ by applying $\triangle Y(n)$-exchanges to all of these $3^{n}$ $n$-tetrahedra. For example, we can see from Fig. \ref{Petersen} that $P^{(1)}$ corresponds to the Petersen graph $P_{10}$. Note that $P_{10}$ is trivalent, meaning that every vertex has degree three. Similarly, we say that a simplicial $n$-complex $K$ is \textit{trivalent} if every $(n-1)$-simplex $\sigma$ in $\varDelta^{n-1}(K)$ has degree ${\rm deg}_{K}^{n}\sigma =3$. Then we have the following.

\begin{Proposition}
For a positive integer $n$, $P^{(n)}$ is trivalent. 
\end{Proposition}

\begin{proof}
The $(n-1)$-simplices in $\varDelta^{n-1}(P^{(n)})$ are of the following three types: 

\vspace{0.2cm}
\noindent
(1) those of the form $|a_{j_{0}}^{0}\cdots \hat{a}_{j_{q}}^{q}\cdots a_{j_{n}}^{n}|$, 

\vspace{0.2cm}
\noindent
(2)  those of the form $|ba_{j_{0}}^{0}\cdots \hat{a}_{j_{p}}^{p}\cdots \hat{a}_{j_{q}}^{q}\cdots a_{j_{n}}^{n}|$, 

\vspace{0.2cm}
\noindent
(3) those generated by each of the $\triangle Y(n)$-exchanges. 

\vspace{0.2cm}
\noindent

First, let us consider $(n-1)$-simplexes of type (1). For the index $j_{q}$ satisfying $\sum_{i=0}^{n}j_{i}\equiv 0\pmod{3}$, the $\triangle Y(n)$-exchange at $\triangle^{n}=\partial |ba_{j_{0}}^{0}\cdots a_{j_{q}}^{q}\cdots a_{j_{n}}^{n}|$ has been applied. Let $x$ be the central $0$-simplex of the corresponding $Y^{n}$. Then the $(n-1)$-simplex $|a_{j_{0}}^{0}\cdots \hat{a}_{j_{q}}^{q}\cdots a_{j_{n}}^{n}|$ is shared by exactly three $n$-simplices: 
\begin{eqnarray*}
|a_{j_{0}}^{0}\cdots a_{j'_{q}}^{q}\cdots a_{j_{n}}^{n}|,\ |a_{j_{0}}^{0}\cdots a_{j''_{q}}^{q}\cdots a_{j_{n}}^{n}|,\ |a_{j_{0}}^{0}\cdots \hat{a}_{j_{q}}^{q}\cdots a_{j_{n}}^{n}x|, 
\end{eqnarray*}
where $j_{q},j'_{q}$ and $j''_{q}$ are mutually distinct. 

Next, let us consider $(n-1)$-simplexes of type (2). Note that the $n$-simplices in $\varDelta^{n}(P^{(n)})$ of the form $|ba_{j_{0}}^{0}\cdots \hat{a}_{j_{q}}^{q}\cdots a_{j_{n}}^{n}|$ do not exist, as the $\triangle Y(n)$-exchange at $\partial |ba_{j_{0}}^{0}\cdots a_{j_{q}}^{q}\cdots a_{j_{n}}^{n}|$ has been applied for the index $j_{q}$ satisfying $\sum_{i=0}^{n}j_{i}\equiv 0\pmod{3}$. Let us take the indices $j_{p}$ and $j_{q}$ satisfying $\sum_{i=0}^{n}j_{i}\equiv 0\pmod{3}$. There are exactly three such pairs $(j_{p},j_{q})$, $(j'_{p},j'_{q})$ and $(j''_{p},j''_{q})$. Consider the $\triangle Y(n)$-exchanges at $\triangle^{n}=\partial |ba_{j_{0}}^{0}\cdots a_{j_{p}}^{p}\cdots a_{j_{q}}^{q}\cdots a_{j_{n}}^{n}|$, $\partial |ba_{j_{0}}^{0}\cdots a_{j'_{p}}^{p}\cdots a_{j'_{q}}^{q}\cdots a_{j_{n}}^{n}|$ and $\partial |ba_{j_{0}}^{0}\cdots a_{j''_{p}}^{p}\cdots a_{j''_{q}}^{q}\cdots a_{j_{n}}^{n}|$. Let $x,x'$ and $x''$ denote the central $0$-simplices of the corresponding $Y^{n}$, respectively. Then the $(n-1)$-simplex $|ba_{j_{0}}^{0}\cdots \hat{a}_{j_{p}}^{p}\cdots \hat{a}_{j_{q}}^{q}\cdots a_{j_{n}}^{n}|$ is shared by exactly three $n$-simplices: 
\begin{eqnarray*}
|ba_{j_{0}}^{0}\cdots \hat{a}_{j_{p}}^{p}\cdots \hat{a}_{j_{q}}^{q}\cdots a_{j_{n}}^{n}x|,\ 
|ba_{j_{0}}^{0}\cdots \hat{a}_{j_{p}}^{p}\cdots \hat{a}_{j_{q}}^{q}\cdots a_{j_{n}}^{n}x'|,\ 
|ba_{j_{0}}^{0}\cdots \hat{a}_{j_{p}}^{p}\cdots \hat{a}_{j_{q}}^{q}\cdots a_{j_{n}}^{n}x''|. 
\end{eqnarray*}

Finally, let us consider $(n-1)$-simplices of type (3). These are of the form $\tau*x$, where $\tau$ is an $(n-2)$-simplex in $\varDelta^{n-2}(\triangle^{n})$. Let $a, a'$ and $a''$ be the $0$-simplices in $\varDelta^{0}(\triangle^{n})$ that are not contained in $\tau$. Then, the $(n-1)$-simplex $\tau*x$ is shared by exactly three $n$-simplices: $\tau*x*a$, $\tau*x*a'$ and $\tau*x*a''$. 
\end{proof}

\section*{Acknowledgment}

The author was supported by JSPS KAKENHI Grant Number 22K03297.

{\normalsize
}

\end{document}